\documentclass[12pt,a4paper,oneside,notitlepage]{article}

\usepackage[T1]{fontenc}
\usepackage[cp1250]{inputenc}
\usepackage{amsfonts, amsmath, amssymb, amsthm}
\usepackage[english]{babel}
\usepackage[bookmarks,colorlinks]{hyperref}
\title{Estimates of perturbation series for kernels
\footnotetext{
The research was partially supported by grants MNiSW N N201 397137, MNiSW N N201 422539 and ANR-09-BLAN-0084-01.}}
\author{Krzysztof Bogdan
\footnote{Institute of Mathematics of the
Polish Academy of Sciences, ul. Śniadeckich 8, 00-956 Warszawa, Poland,
bogdan@pwr.wroc.pl}
 \and Tomasz Jakubowski
\footnote{Institute of Mathematics and Computer Science, Wroc{\l}aw University of Technology, Wybrze\.ze Wyspia\'nskiego 27, 50-370 Wroc{\l}aw, Poland,
Tomasz.Jakubowski@pwr.wroc.pl} \and Sebastian Sydor\footnote{Institute of Mathematics, University of Wroc{\l}aw, pl. Grunwaldzki 2/4, 50-384 Wroc{\l}aw, Poland, Sebastian.Sydor@math.uni.wroc.pl}}
\date{\today}

\newtheorem{theorem}{Theorem}
\newtheorem{lemma}{Lemma}
\newtheorem{definition}{Definition}
\newtheorem{remark}{Remark}

\newtheorem{example}{Example}

\def \si {\sigma }

\def \bbR {{\mathbb R}}
\def \bbN {{\mathbb N}}
\def \cM {{\cal M}}
\def \cB {{\cal B}}
\def \cE {{\cal E}}

\def \intl {\int}
\def \t {\tilde }
\def \qq {{q}}
\def \kK {{K}}
\def \kk {{k}}
\def \ka {{\kappa}}

\begin{document}

\maketitle

\begin{abstract}
For integral kernels on space-time we indicate a class of nonnegative Schr\"odinger perturbations which produce comparable integral kernels.
\end{abstract}
\section{Introduction}
Schr\"odinger operators $\Delta+q$ were studied for the Laplacian $\Delta$, e.g., in 
\cite{MR936811, MR1329992, MR1642818, MR1978999}.
Local integral smallness of the function $q$, defined as a Kato-type condition (\cite{MR1329992,MR1978999})
played an important role in these considerations.
Similar Schr\"odinger operators
based on the fractional Laplacian $\Delta^{\alpha/2}$ were studied in
\cite{MR1473631, MR1671973, MR1825645} 
(see also \cite{MR1920109}), with focus on {\it comparability}
of the resulting Green functions.
The corresponding estimates for general transition densities were then studied in \cite{MR2457489} under the following integrability condition on $q$,
\begin{equation}\label{eq1}
\int\limits^{t}_{s}\int\limits_X p(s,x,u,z) |q(u,z)| p(u,z,t,y) dzdu \leq [\eta+\beta(t-s)] p(s,x,t,y),
\end{equation}
where $p$ is a finite jointly measurable transition density, $\beta$ and $\eta$ are fixed nonnegative numbers, while times $s<t$ and states $x,y$ are arbitrary. Given (\ref{eq1}), the following estimate was obtained in \cite{MR2457489},
\begin{equation}\label{eq2}
\tilde p(s,x,t,y)\leq \frac{1}{1-\eta}\exp{\left(\frac{\beta}{1-\eta}(t-s)\right)}p(s,x,t,y),
\end{equation}
provided $\eta <1$. Here $\tilde p$ denotes the Schr\"odinger perturbation series defined by $p$ and $q$ (see below for details).
The approach of \cite{MR2457489} depends on nontrivial combinatorics of the perturbation series. Further combinatorial arguments were used  in \cite{MR2507445} to refine the above result by skipping the Chapman-Kolmogorov condition on $p$, relaxing the assumptions on $q$, and strengthening the estimate, as in (\ref{eq:metgKpt}) below. Meanwhile, a more straightforward method was proposed in \cite{MR2643799} for gradient perturbations of the transition density of the fractional Laplacian.
As suggested in \cite{MR2643799}, the technique extends to Schr\"odinger perturbations, and yields the main results of \cite{MR2507445}.
We present here the extension, which also allows to perturb Markovian semigroups, potential kernels, and in fact  general {\it forward} integral kernels on space-time by rather singular functions $q$.

We like to mention a related paper \cite{2011-BHJ} on the von Neumann series of general integral kernels with a certain transience-type property. Both papers were inspired by \cite{MR2457489,MR2507445}, but their methods and results are different. The present estimates are more convenient and specific
for forward kernels in {continuous time} perturbed by {functions}.

In what follows we will assume that $q$ is nonnegative, since the absolute value of the perturbation with signed $q$ is bounded by the perturbation with $|q|$, if finite.
In this connection we also note that a discussion of the {\it positive} lower bound for signed perturbations of transition densities is given in \cite{MR2457489}.

Our main results are given in Section~\ref{sec:ik}. Examples of applications and further comments are given in  Section~\ref{sec:ap}. 
In particular, we estimate the inverse kernel of Schr\"odinger perturbations of Weyl fractional derivatives on the real line.  
\section{Preliminaries}
We will recall, after \cite{MR939365}, basic properties of kernels.
\begin{definition}
Let $(E,\cE)$ be a measurable space. A kernel on $E$ is a map $K$ from $E\times\cE$ to $[0,\infty]$ with the following properties:
\begin{enumerate}
\item $x\mapsto K(x,A)$ is $\cE$-measurable for all $A\in\cE$,
\item $A\mapsto K(x,A)$ is countably additive for all $x\in E$.
\end{enumerate}
\end{definition}
\noindent Consider kernels $K$ and $L$ on $E$. The map
$$
(x,A)\mapsto \intl_E K(x,dy)L(y,A)
$$
from $(E\times\cE)$ to $[0,\infty]$ is a kernel on $E$, called the {\it composition} of $K$ and $L$, and denoted $KL$. Composition of kernels is associative (\cite{MR939365}).
We write $q\in \cE^+$ if $q:E\to [0,\infty]$ and $q$ is $\cE$-measurable.
We will denote by the same symbol the {\it kernel}
$q(x,A)=q(x)1_A(x)$. Here $1_A$ is the indicator function of $A$.
We let $\kK _n=(\kK q)^n\kK$, $n=0,1,\ldots$.
Associativity yields the following.
\begin{lemma}\label{l1}
$\kK _n=\kK _{n-1-m}\qq \kK _m$ for all $n\in\bbN$ and $m=0,1,\ldots,n-1$.
\end{lemma}
We will consider the {\it perturbation} of ${K}$ by $q$, defined as the kernel
\begin{equation}\label{eq:deftk}
\t{K}=\sum_{n=0}^\infty \kK_n=\sum_{n=0}^\infty(Kq)^nK.
\end{equation}
Of course, $K\leq \t{K}$. In what follows we will prove upper bounds for $\t{K}$ under additional conditions on $K$ and $K_1=KqK$. 

\section{Estimates for kernels on space-time}\label{sec:ik}
In what follows we consider a set $X$ (the state space) with $\si$-algebra $\cM$, the real line  $\bbR$ (the time) equipped with the Borel sets $\cB_\bbR$, and $E=\bbR \times X$ (the space-time) with the product $\si$-algebra ${\cE}=\cB_{\bbR}\times \cM$.
We also fix  $q\in \cE^+$, a number $\eta \in [0,\infty)$ and a function $Q: \bbR \times \bbR \rightarrow [0, \infty)$ satisfying the following condition of {\it super-additivity}:
\begin{equation}\label{a2}
Q(u,r)+Q(r,v)\leq Q(u,v)\quad \mbox{for all }u<r<v.
\end{equation}
Let $\kK $
be a kernel on $E$. We will assume that $K$ is a {\it forward} kernel, i.e.
\begin{equation}\label{kernel}
\kK (s,x,A)=0\quad \mbox{provided } A\subseteq(-\infty,s]\times X \quad (A\in \cE,\; s\in \bbR).
\end{equation}
\begin{remark}
{\rm 
In the language of \cite{2011-BHJ}, $(s,\infty)\times X$ is {\it absorbing} for forward kernels.
}
\end{remark}
\noindent
We will also assume that
\begin{equation}\label{a1}
\kK \qq \kK (s,x,A)\leq \intl_{A}\kK (s,x,dtdy)\left[\eta+Q(s,t)\right],\quad 
(s,x)\in E, A\in \cE.
\end{equation}
From now on (\ref{a1}) and similar inequalities will be abbreviated as follows,
\begin{equation}\label{eq:aa1}
\kK_1(s,x,dtdy)\leq \kK (s,x,dtdy)\left[\eta+Q(s,t)\right].
\end{equation}
\begin{theorem}\label{t1}
For all $n=1,2,\ldots$, and $(s,x)\in E$,
\begin{align}
\kK _n(s,x,dtdy)&\leq \kK _{n-1}(s,x,dtdy)\left[\eta + \frac{Q(s,t)}{n}\right]\label{e1}\\
&\leq \kK (s,x,dtdy)\prod_{k=1}^{n}\left[\eta + \frac{Q(s,t)}{k}\right].\label{e2}
\end{align}
If $0<\eta <1$, then for all $(s,x)\in E$,
\begin{equation}\label{eq:metK}
\t{\kK} (s,x,dtdy)\leq \kK (s,x,dtdy){\left(\frac{1}{1-\eta}\right)}^{1+Q(s,t)/\eta}.
\end{equation}
If $\eta =0$, then for all $(s,x)\in E$,
\begin{equation}\label{e4}
\t{\kK}(s,x,dtdy)\leq\kK (s,x,dtdy) e^{Q(s,t)}.
\end{equation}
\end{theorem}
\begin{proof}
(\ref{a1}) gives (\ref{e1}) for $n=1$. By induction, Lemma \ref{l1}, (\ref{a1}) and (\ref{kernel}),
\begin{align*}
 (n+1)\kK _{n+1}(s,x,A)=&n\kK _n\qq \kK (s,x,A)+\kK _{n-1}\qq \kK _1(s,x,A)\\
\leq& n\intl_{E}\kK _{n-1}(s,x,dudz)\left[\eta + \frac{Q(s,u)}{n}\right]\qq (u,z)\kK (u,z,A)\\
+&\intl_A\intl_{E}\kK _{n-1}(s,x,dudz)\qq (u,z)\kK (u,z,dtdy)[\eta+Q(u,t)]\\
\leq&
\intl_A \kK _n(s,x,dtdy)\left[(n+1)\eta + Q(s,t)\right],
\end{align*}
as needed. (\ref{e2}) follows from (\ref{e1}), (\ref{e4}) results from Taylor\rq{}s expansion of the exponential function, and (\ref{eq:metK}) follows from the Taylor series
$$
(1-\eta)^{-a}=\sum_{n=0}^\infty \frac{\eta^n(a)_n}{n!},
$$
where $0<\eta<1$, $a\in \bbR$, and $(a)_n=a(a+1)\cdots(a+n-1)$.
\end{proof}

Theorem~\ref{t1} has two {\it fine} or {\it pointwise} variants, which we will state under suitable conditions.
We fix a (nonnegative) $\si$-finite, non-atomic measure $$dt=\mu(dt)$$ on $(\bbR,\cB_\bbR)$ and a function $\kk(s,x,t,A)$ defined for $s<t$, $x\in X$, $A \in \cM$, such that $(s,x,t)\mapsto \kk(s,x,t,A)\in [0,\infty)$ is jointly measurable. 
We will call $k$ a transition kernel if it satisfies the Chapman-Kolmogorov conditions, see (\ref{eq:CK}).
For instance, if $p$ is a transition probability, and we let $k(s,x,t,A)=p_{s,t}(x,A)$, then $k$ is a transition kernel, provided it is jointly measurable.
We let $\kk_0=\kk$, and for $n=1,2,\ldots$, we define
\begin{align*}
\kk_n(s,x,t,A)&=\int_s^t\int_X \kk_{n-1}(s,x,u,dz)q(u,z)\kk(u,z,t,A)du.
\end{align*}
\begin{lemma}\label{lem:knft} If $n\in\bbN$, $m=0,1,\ldots,n-1$, $s<t$, $x\in X$ and $A\in \cE$, then
\begin{equation}\label{eq:pnmA}
\kk _n(s,x,t,A)=\int_s^t \int_X \kk _{n-1-m}(s,x,u,dz)\qq(u,z) \kk _m(u,z,t,A)du.
\end{equation}
\end{lemma}
\begin{proof}
If $m=0$, then the equality (\ref{eq:pnmA}) holds by the definition of $k_n$. In
particular, this proves our claim for $n=1$. If $n\geq 1$ is
such that (\ref{eq:pnmA}) holds for all $m<n$, then so for every 
$m=1,2,\ldots,n$, we obtain
\begin{eqnarray*}
&&  k_{n+1}(s,x,t,A)=\intl_s^t \intl_ X k_{n}(s,x,u,dz) q(u,z) k(u,z,t,A) du \\
&=& \intl_s^t \intl_ X \intl_s^u \intl_X k_{n-1-(m-1)}(s,x,v,dz_1) q(v,z_1) k_{m-1}(v,z_1,u,dz)  dv \\
&&\times q(u,z)k(u,z,t,A) du\\
&=& \intl_s^t \intl_ X k_{n-m}(s,x,v,dz_1) q(v,z_1) \\
&&\times \left( \intl_v^t \int_ X k_{m-1}(v,z_1,u,dz) q(u,z) k(u,z,t,A) du \right) dv \\
&=& \intl_s^t \intl_ X k_{n-m}(s,x,v,dz_1) q(v,z_1) k_{m}(v,z_1,t,A)  dv.
\end{eqnarray*}
\end{proof}

We define
\begin{equation}\label{eq:defpttk}
\t{k}=\sum_{n=0}^\infty \kk_n.
\end{equation}
We will assume that for all $s \leq t \in \bbR$, $x\in X$ and $A\in \cM$, 
\begin{equation}\label{a3}
\intl_s^t\intl_X \kk(s,x,u,dz) q(u,z)\kk(u,z,t,A)du\leq [\eta+Q(s,t)]\kk(s,x,t,A),
\end{equation}
or $\kk_1(s,x,t,dy)\leq [\eta+Q(s,t)]\kk(s,x,t,dy)$.
Thus, (\ref{a3}) is a {\it fine} version of (\ref{a1}). %
\begin{theorem}\label{t2}
For all $n=1,2,\ldots$, $s<t$ and $x\in X$,
\begin{align}\label{eq:pte1}
\kk _n(s,x,t,dy)&\leq \kk _{n-1}(s,x,t,dy)\left[\eta + \frac{Q(s,t)}{n}\right],\\
\label{eq:pte2}
&\leq \kk (s,x,t,dy)\prod_{k=1}^{n}\left[\eta + \frac{Q(s,t)}{k}\right].
\end{align}
If $0<\eta <1$, then for all $s<t$ and $x\in X$,
\begin{equation}\label{eq:metKpt}
\t{\kk} (s,x,t,dy)\leq \kk (s,x,t,dy){\left(\frac{1}{1-\eta}\right)}^{1+Q(s,t)/\eta}.
\end{equation}
If $\eta =0$, then for all $s<t$ and $x\in X$,
\begin{equation}\label{eq:pte4}
\t{\kk}(s,x,t,dy)\leq \kk (s,x,t,dy)e^{Q(s,t)}.
\end{equation}
\end{theorem}
\begin{proof}
By Lemma~\ref{lem:knft}, induction, (\ref{a3}) and (\ref{a2}), for $n\geq 1$ we have
\begin{align*}
&(n+1)\kk_{n+1}(s,x,t,A)\\
&\leq n\intl_s^t\intl_X\kk_{n-1}(s,x,u,dz)\left[\eta + \frac{Q(s,u)}{n}\right]q(u,z)\kk(u,z,t,A)du\\
&+\intl_s^t\intl_X\kk_{n-1}(s,x,u,dz)q(u,z)\kk(u,z,t,A)\left[\eta + \frac{Q(u,t)}{n}\right]du\\
&=(n+1)\left[\eta + \frac{Q(s,t)}{n+1}\right]\kk_n(s,x,t,A),\quad A\in \cM.
\end{align*}
For $n=1$, (\ref{eq:pte1}) is identical with (\ref{a3}). We proceed as in Theorem~\ref{t1}.
\end{proof}

For the {\it finest} variant of Theorem~\ref{t1}, we fix a 
$\si$-finite measure 
$$dz=m(dz)$$ 
on $(X, \cM)$.
We will consider function $\ka(s,x,t,y)$ defined for $s<t$ and $x,y\in X$, such that $(s,x,t,y)\mapsto \ka(s,x,t,y)\in [0,\infty)$
is $\cB_{\bbR} \times \cM \times \cB_{\bbR}\times \cM$-measurable. 
We will call such $\ka$ a (forward) kernel density, because $\int_{\{(t,y)\in E: s<t\}} \ka(s,x,t,y)f(t,y)dtdy$ is a forward kernel on $E$.
For instance, we may take $k(s,x,t,y)=p_{s,t}(x,y)$, if measurable and finite, where $p$ is a transition probability density function.
We define $\ka_0(s,x,t,y) = \ka(s,x,t,y)$, 
$$
\ka_n(s,x,t,y) = \intl_s^t \intl_X \ka_{n-1}(s,x,u,z)q(u,z)\ka(u,z,t,y) \,dz\,du\,,\quad n=1,2,\ldots\,.
$$

\begin{lemma}\label{lem:pnm}
For all $n=1,2,\ldots$, $m=0,1,\ldots,n-1$, $s,t\in \bbR$ and $x,y \in  X$, 
\begin{equation} \label{eq:pnm}
 \ka_n(s,x,t,y)  =  \intl_s^t\intl_{X} \ka_{n-1-m}(s,x,u,z) q(u,z) \ka_{m}(u,z,t,y) dzdu.
\end{equation}
\end{lemma}
\begin{proof}
The result was stated in \cite[Lemma~3]{MR2457489} under stronger conditions, so for the comfort of the reader we repeat the arguments of \cite{MR2457489}.

If $m=0$, then the equality (\ref{eq:pnm}) holds by the definition of $\ka_n$. In
particular, this proves our claim for $n=1$. If $n\geq 1$ is
such that (\ref{eq:pnm}) holds for all $m<n$, then for every 
$m=1,2,\ldots,n$, by Fubini we indeed obtain
\begin{eqnarray*}
&&  \ka_{n+1}(s,x,t,y)=\intl_s^t \intl_ X \ka_{n}(s,x,u,z) q(u,z) \ka(u,z,t,y) dzdu \\
&=& \intl_s^t \intl_ X \intl_s^u \intl_X \ka_{n-1-(m-1)}(s,x,v,z_1) q(v,z_1) \ka_{m-1}(v,z_1,u,z) dz_1 dv \\
&&\times q(u,z)\ka(u,z,t,y) dzdu\\
&=& \intl_s^t \intl_ X \ka_{n-m}(s,x,v,z_1) q(v,z_1) \\
&&\times \left( \intl_v^t \int_ X \ka_{m-1}(v,z_1,u,z) q(u,z) \ka(u,z,t,y) dzdu \right) dz_1 dv \\
&=& \intl_s^t \intl_ X \ka_{n-m}(s,x,v,z_1) q(v,z_1) \ka_{m}(v,z_1,t,y) dz_1 dv.
\end{eqnarray*}
\end{proof}
The  Schr\"odinger perturbation of $\ka$ by $q$ is defined as follows,
\begin{equation}\label{eq:tp}
\t{\ka}=  \sum_{n=o}^\infty \ka_n.
\end{equation}
We will assume that for all $s<t\in \bbR$ and $x,y\in X$,
\begin{equation}\label{eq:kk_1}
\intl_s^t\intl_X\ka(s,x,u,z)q(u,z)\ka(u,z,t,y)dzdu\leq [\eta +Q(s,t)]\ka(s,x,t,y),
\end{equation}
or $\ka_1(s,x,t,y)\leq \ka(s,x,t,y)[\eta+Q(s,t)]$.
This is a {\it fine} analogue of (\ref{a1}) and ({\ref{a3}).
The following is a {\it fine} version of Theorem~\ref{t1} and \ref{t2}.
We note that (\ref{eq:gpte2}, \ref{eq:metgKpt}, \ref{eq:gpte4}), but not (\ref{eq:gpte1}), were first proved in \cite{MR2507445} by involved combinatorics. 
\begin{theorem}\label{t3}
For all $n=1,2,\ldots$, $s<t$ and $x,y \in X$,
\begin{align}\label{eq:gpte1}
\ka _n(s,x,t,y)&\leq \ka _{n-1}(s,x,t,y)\left[\eta + \frac{Q(s,t)}{n}\right]\\
\label{eq:gpte2}
&\leq \ka (s,x,t,y)\prod_{k=1}^{n}\left[\eta + \frac{Q(s,t)}{k}\right].
\end{align}
If $0<\eta <1$, then for all $s<t$ and $x,y \in X$, 
\begin{equation}\label{eq:metgKpt}
\t{\ka} (s,x,t,y)\leq \ka (s,x,t,y){\left(\frac{1}{1-\eta}\right)}^{1+Q(s,t)/\eta}.
\end{equation}
If $\eta =0$, then for all $s<t$ and $x,y \in X$,
\begin{equation}\label{eq:gpte4}
\t{\ka}(s,x,t,y)\leq \ka (s,x,t,y)e^{Q(s,t)}.
\end{equation}
\end{theorem}
\begin{proof} We proceed as in the proof of Theorem~\ref{t1}, 
using Lemma~\ref{lem:pnm} and (\ref{eq:kk_1}).
\end{proof}

\section{Discussion and Applications}\label{sec:ap}
The proofs of Theorem~\ref{t1}, ~\ref{t2} and \ref{t3} indicate that our estimates are rather tight.  The observation is supported by the exact formulas for 
Schr\"odinger perturbations of transition densities by Dirac measures (not directly manageable by the methods of the present paper), see \cite{2011-BHJ}.
 We like to note that the iterated integrals defining $\kK_n$, $\kk_n$ and $\ka_n$ exhibit similarity to the expectations of powers of the additive functional in Khasminski\rq{}s lemma (\cite{MR1329992}, \cite{MR1671973}), to Wiener chaoses and the multiple integrals in the theory of rough paths (\cite{MR2314753}). In fact, our results offer a far-reaching extension and strengthening of Khasminski\rq{}s lemma for transition kernels and densities. On a formal level, a unique feature of our estimates is the combinatorics triggered by $\eta$, $Q$ and the assumptions (\ref{a1}), (\ref{a3}), (\ref{eq:kk_1}). As we will see below, the presence of $\eta$ is quite convenient in applications, and $Q$ is often chosen linear.

In applications, we need to verify conditions (\ref{a1}), (\ref{a3}) or (\ref{eq:kk_1}).
\begin{example}
{\rm 
Let $k(s,x,t,dy)\geq 0$ be a (jointly measurable) {\it transition} kernel, so that the following Chapman-Kolmogorov identity holds for all $A\in \cM$, $x\in X$ and $s<u<t$,
\begin{equation}\label{eq:CK}
\int_X k(s,x,u,dz)k(u,z,t,A)=k(s,x,t,A).
\end{equation}
If $du$ is the linear Lebesgue measure and $\|q\|_\infty:=\sup |q(u,z)|<\infty$, then
$$
k_1(s,x,t,A)\leq \|q\|_\infty\, k(s,x,t,A)\int_s^t du.
$$
Theorem~\ref{t2}, $Q(s,t)=\|q\|_\infty (t-s)$ and $\eta=0$ yield the well-expected bound,
\begin{equation}\label{eq:tkbq}
\t{\kk}(s,x,t,dy)\leq \kk (s,x,t,dy)e^{\|q\|_\infty(t-s)}.
\end{equation}
By Theorem~\ref{t3}, an analogous pointwise version of (\ref{eq:tkbq}) also holds.
}
\end{example}

\begin{example}\label{ex:osrk}
{\rm 
If $X=\{x_0\}$ consists of only one point and $dz$ is the Dirac measure at $x_0$, then we can skip them from the notation.
For instance, let $0<\beta<1$, $s<t$, and 
$\ka(s,t)=\Gamma(\beta)^{-1}(t-s)^{\beta-1}$.
For the linear Lebesgue measure $du$, Borel function $u\mapsto q(u)\geq 0$ and $s<t$, 
\begin{align}
\ka_1(s,t)&=\frac{1}{\Gamma(\beta)^2}\int_s^t (u-s)^{\beta-1} q(u) (t-u)^{\beta-1}du\label{eq:defka1}\\
&\le \frac{\|q\|_\infty} {\Gamma(2\beta)}(t-s)^{2\beta-1} 
=\frac{\|q\|_\infty} {\Gamma(2\beta)}(t-s)^{\beta}\ka(s,t) \nonumber\\
&\le [\eta+c(t-s)]\ka(s,t),\label{eq:lb}
\end{align}
provided $\|q\|_\infty<\infty$. Here $\eta>0$ may be arbitrarily small, at the expense of $c<\infty$. We note that such affine upper bounds are an important special case of (\ref{eq:kk_1}), in particular (\ref{eq:lb}) allows for an application of Theorem~\ref{t3}.

We can handle some unbounded functions $q$, too. For $s<u<t$ we have
$$
(u-s)^{1-\beta}\vee(t-u)^{1-\beta}\geq \left[(t-s)/2\right]^{1-\beta},
$$
hence the following 3P Theorem holds for $\ka$,
$$
\ka(s,u)\wedge \ka(u,t)\le 2^{1-\beta}\ka(s,t).
$$
In consequence, $\ka(s,u)\ka(u,t)\le 2^{1-\beta}\ka(s,t)\left[
\ka(s,u)+\ka(u,t)\right]$. By (\ref{eq:defka1}),
\begin{equation}
\ka_1(s,t)\leq \ka(s,t)\frac{2^{1-\beta}}{\Gamma(\beta)}
\big[
\int_s^t (u-s)^{\beta-1}q(u)du+\int_s^t(t-u)^{\beta-1} q(u)du
\big].\label{eq:uKc}
\end{equation}
In particular, $q(u)=|u|^{-\beta+\varepsilon}$ with $0<\varepsilon\le\beta$, yields sufficient smallness of the integrands in (\ref{eq:uKc}), hence local comparability of $\ka$ and $\t \ka$, by Theorem~\ref{t3}. 
}
\end{example}
\begin{remark}
{\rm 
Let $\ka$ be a (forward) kernel density.
We will say that $q$ is of {\it relative Kato class} for $\ka$, if $\inf\{c:\int_s^t\int_X \ka(s,x,u,z)q(u,z)\ka(u,z,t,y)dzdu\leq 
c\ka(s,x,t,y) \mbox{ for all  $s<t<s+h$ and $x,y\in X$}\}\to 0$ as $h\to 0$.
In short,
$$\sup\{\ka_1(s,x,t,y)/\ka(s,x,t,y): s<t<s+h,\ x,y\in X\big\}\to 0 \mbox{ as }  h\to 0.$$
We say that $q$ is of {\it Kato class} for $\ka$, if 
$$\sup\big\{\int_s^t\int_X \left[\ka(s,x,u,z)+\ka(u,z,t,y)\right]q(u,z)dzdu\big\}\to 0 \mbox{ as } h\to 0,$$
where the supremum is taken over all $s<t<s+h$ and  $x,y\in X$. 
The conditions were used for Schr\"odinger perturbations of {transition densities}, for which the latter 
is usually weaker and easier to verify, see \cite{MR2457489}.
As indicated by Example~\ref{ex:osrk}, when $\ka$ satisfies the 3P Theorem, the  Kato condition implies the relative Kato condition.
Accordingly, the two  are equivalent for the transition density of the fractional Laplacian $\Delta^{\alpha/2}$ with $0<\alpha<2$, but not $\alpha=2$, because 3P fails for the Gaussian kernel. The details and
 further references are given in \cite{MR2457489}  for transition densities, see also \cite{2011-BHJ} for the special case of Schr\"odinger perturbations of the Cauchy transition density. 
}
\end{remark}

We will make a connection to Schr\"odinger operators analogous to $\Delta+q$, as aforementioned in Introduction. 
Consider a kernel $\kK$ on $E$, function $q\in \cE^+$ and real-valued $\cE$-measurable functions $\phi$ and $\psi$ on $E$ such that $\kK \psi=-\phi$. Here we assume absolute integrability: $\kK|\psi|<\infty$.
Then, 
\begin{align}
\t\kK (\psi+q\phi)&=(\kK+\t\kK q\kK)(\psi+q\phi)
=-\phi +\kK q \phi-\t \kK q \phi+\t \kK q \kK q \phi \nonumber\\
&=-\phi +\kK q \phi-\kK q \phi-\t \kK q\kK q \phi+\t \kK q \kK q \phi=-\phi,\label{eq:tkphi}
\end{align}
provided the integrals are absolutely convergent for all arguments.

For forward kernels we can give rather explicit sufficient conditions for the absolute integrability.
We will say $\kK$ is locally finite in time if for all real $s<t$, $u\in \bbR$ and $z\in X$, we have
$\kK 1_{(s,t)}(u,z)=\kK(u,z,(s,t)\times X)<\infty$.
\begin{lemma}\label{thm:li}
Consider a forward kernel $\kK$ locally finite in time. Let $q\in \cE^+$ satisfy {\rm (\ref{a1})} with $\eta<1$ and some superadditive function $Q$. Let $\psi$ and $\phi$ be real-valued $\cE$-measurable functions such that $\kK\psi=-\phi$, and $|\psi|\leq c 1_{(a,b)}$ for some $a,b,c\in \bbR$. Then $\t\kK (\psi+q\phi)=-\phi$.
\end{lemma}
\begin{proof}
We have $|\phi|\leq \kK|\psi|<\infty$, by the local finiteness of $\kK$. By the preceding discussion it suffices to prove that $\kK q\kK|\psi|$, $\t \kK q\kK|\psi|$ and $\t \kK q \kK q \kK |\psi|$ are finite.
In {\it bounded time}, by our assumptions and Theorem~\ref{t1}, $\kK q\kK\leq C \kK$, $\t \kK\leq C \kK$, and $\kK q\kK q \kK\leq C \kK$, with some $C\in \bbR$, which ends the proof.
\end{proof}

As a rule, if $\kK$ is a left inverse of an operator $L$ on space-time, then  $\t\kK$
is a left inverse of $L +q$.
Namely, if
$$
\int_E \kK(s,x,dudz)L\phi(u,z)=-\phi(s,x), \quad (s,x)\in E,
$$
for some function $\phi$, 
then we consider $\psi=L\phi$, and  obtain
$$
\int_E \t\kK(s,x,dudz)\left[L\phi(u,z)+q(u,z)\phi(u,z)\right]=-\phi(s,x), \quad (s,x)\in E,
$$
under the assumptions of Lemma~\ref{thm:li}.
This is quite satisfactory if $L$ is local in time, because if $\phi$ is compactly supported in time, then so is $\psi$,  and the boundedness of $\psi$ may usually be secured by appropriate assumptions on $\phi$, see, e.g., \cite{MR2457489, 2011-KB-TJ-pa}.

If $L$ is nonlocal in time, then more flexible conditions on $\kK$ may be needed.
\begin{lemma}\label{thm:li2}
Consider a forward kernel $\kK$ such that $\kK^2$ is locally finite in time. Let $q\in \cE^+$ satisfy {\rm (\ref{a1})} with $\eta<1$ and some superadditive function $Q$. Let $\psi$ and $\phi$ be real-valued $\cE$-measurable functions such that $\kK\psi=-\phi$, and $|\psi|\leq c \kK 1_{(a,b)}$ for some $a,b,c\in \bbR$. Then $\t\kK (\psi+q\phi)=-\phi$.
\end{lemma}
\begin{proof}
The absolute integrability required for (\ref{eq:tkphi}) amounts to the finiteness of
$|\phi|\leq \kK|\psi|$, $\kK q\kK |\psi|$, $\t \kK q\kK|\psi|$ and $\t \kK q \kK q \kK |\psi|$.
In {bounded time}, by Theorem~\ref{t1}, $\kK q\kK\leq C \kK$, $\t \kK\leq C \kK$, and $\kK q\kK q \kK\leq C \kK$, with a number $C$. The result follows, since $\kK^2 1_{(a,b)}<\infty$ for finite $a<b$.
\end{proof}

\begin{example}\rm 
We consider the Weyl fractional integral on the real line (\cite{1993-KM-BR}),
$$
W^{-\beta}\psi(s) = \frac{1}{\Gamma(\beta)} \int_s^\infty (u-s)^{\beta-1} \psi(u)\, du\,.
$$ 
Here $\beta\in (0,1)$, and we require absolute integrability. The kernel has the density $\kappa(s,u) = (u-s)^{\beta-1}/\Gamma(\beta)$ discussed in Example \ref{ex:osrk}. 
We also consider the Weyl fractional derivative,
$$
\partial^{\beta}\phi(s) =  \frac{1}{\Gamma(1-\beta)}\int_s^\infty (u-s)^{-\beta}\phi'(u)\,du\,.
$$
Here and in what follows $s\in \bbR$ and $\phi$ is a real-valued, continuously differentiable and compactly supported function on $\bbR$. 
By Fubini's theorem, 
\begin{align*}
W^{-\beta}\partial^{\beta} \phi(s) &= \frac{1}{\Gamma(\beta)\Gamma(1-\beta)} \int_s^\infty\int_u^\infty (u-s)^{\beta-1}(r-u)^{-\beta} \phi'(r)\, dr\, du \\
& = \frac{1}{\Gamma(\beta)\Gamma(1-\beta)} \int_s^\infty\int_s^r (u-s)^{\beta-1}(r-u)^{-\beta} \phi'(r)\, du\, dr \\
& = \int_s^\infty \phi'(r)\,  dr = -\phi(s),
\end{align*}
see, e.g., \cite{1993-KM-BR}. We intend to use Lemma \ref{thm:li2}.
Let $\psi=\partial^\beta \phi$. If $a,b\in \bbR$ and ${\rm supp}\ \phi \subset (a,b)$, then
$|\psi(s)|\leq(\Gamma(1-\beta))^{-1}\|\phi'\|_\infty\int_0^{b-a} u^{-\beta}du  $ for all $s\in \bbR$, and
$\psi(s)=0$ for $s>b$. Since $\int_a^b\phi'(u)du=0$, for $s<a$ we obtain
$$
\psi(s)=\frac{1}{\Gamma(1-\beta)}\int_a^b \left[(u-s)^{-\beta}-(a-s)^{-\beta}\right]\phi'(u)du,
$$
hence 
$|\psi(s)|\leq (\Gamma(1-\beta))^{-1}\beta(b-a)^2 (a-s)^{-\beta-1}\|\phi'\|_\infty$.
On the other hand, 
$$
W^{-\beta}1_{(a',b')}(s)\geq \frac{b'-a'}{\Gamma(\beta)}(b'-s)^{\beta-1},
$$
if $s<a'<b'<\infty$.
When multiplied by a constant, this majorizes $\psi$, provided $a'>b$.
Since $W^{-\beta}1_{(a',b')}$ is locally bounded, and $W^{-\beta}$ is locally finite, we see that $\left(W^{-\beta}\right)^2$ is locally finite.

We now consider $q\in \cE^+$ satisfying {\rm (\ref{eq:kk_1})} with $\eta<1$ and a superadditive function $Q$ (see  Example \ref{ex:osrk} for such $q$). By Lemma \ref{thm:li2} and the above discussion,
\begin{equation}
\int_s^\infty \tilde\kappa(s,u)\left[\partial^\beta \phi(u)+q(u)\phi(u)\right]\,du=-\phi(s),
\end{equation}
where, by Theorem \ref{t3}, 
\begin{equation}
\tilde{\kappa}(s,t) = \sum_{n=0}^\infty \kappa_n(s,t) \le \frac{1}{\Gamma(\beta)}\left(\frac{1}{1-\eta}\right)^{1+Q(s,t)/\eta}
(t-s)^{\beta-1},
\qquad s<t.
\end{equation}
It seems that our methods also apply to perturbations of the so called anomalous diffusions, which are driven by fractional time derivatives, see \cite{MR2568272, MR2782245, Meerschaert2011216}.

\end{example}

\noindent
{\bf Acknowledgments.} Part of the results were presented in 2011 at 
the workshop “Foundations of Stochastic Analysis” at Banff International Research Station, and at Stochastic Analysis Seminar at Oxford-Man Institute of Quantitative Finance. The first named author gratefully thanks for the invitations. We also thank Wolfhard Hansen and Karol Szczypkowski for useful comments. 

\bibliographystyle{abbrv}
\bibliography{BJS-bib}

\def\polhk#1{\setbox0=\hbox{#1}{\ooalign{\hidewidth
  \lower1.5ex\hbox{`}\hidewidth\crcr\unhbox0}}}
  \def\polhk#1{\setbox0=\hbox{#1}{\ooalign{\hidewidth
  \lower1.5ex\hbox{`}\hidewidth\crcr\unhbox0}}}
\begin{thebibliography}{10}

\bibitem{MR1671973}
K.~Bogdan and T.~Byczkowski.
\newblock Potential theory for the {$\alpha$}-stable {S}chr\"odinger operator
  on bounded {L}ipschitz domains.
\newblock {\em Studia Math.}, 133(1):53--92, 1999.

\bibitem{MR1825645}
K.~Bogdan and T.~Byczkowski.
\newblock Potential theory of {S}chr\"odinger operator based on fractional
  {L}aplacian.
\newblock {\em Probab. Math. Statist.}, 20(2, Acta Univ. Wratislav. No.
  2256):293--335, 2000.

\bibitem{MR2457489}
K.~Bogdan, W.~Hansen, and T.~Jakubowski.
\newblock Time-dependent {S}chr\"odinger perturbations of transition densities.
\newblock {\em Studia Math.}, 189(3):235--254, 2008.

\bibitem{2011-BHJ}
K.~{Bogdan}, W.~{Hansen}, and T.~{Jakubowski}.
\newblock Localization and {S}chr{\"o}dinger perturbations of kernels.
\newblock Preprint (arXiv), 2012.

\bibitem{2011-KB-TJ-pa}
K.~Bogdan and T.~Jakubowski.
\newblock Estimates of the {G}reen function for the fractional {L}aplacian
  perturbed by gradient.
\newblock {\em Potential Analysis}, pages 1--27, June 2011.
\newblock Online First: DOI 10.1007/s11118-011-9237-x.

\bibitem{MR1473631}
Z.-Q. Chen and R.~Song.
\newblock Intrinsic ultracontractivity and conditional gauge for symmetric
  stable processes.
\newblock {\em J. Funct. Anal.}, 150(1):204--239, 1997.

\bibitem{MR1920109}
Z.-Q. Chen and R.~Song.
\newblock General gauge and conditional gauge theorems.
\newblock {\em Ann. Probab.}, 30(3):1313--1339, 2002.

\bibitem{MR1329992}
K.~L. Chung and Z.~X. Zhao.
\newblock {\em From {B}rownian motion to {S}chr\"odinger's equation}, volume
  312 of {\em Grundlehren der Mathematischen Wissenschaften [Fundamental
  Principles of Mathematical Sciences]}.
\newblock Springer-Verlag, Berlin, 1995.

\bibitem{MR936811}
M.~Cranston, E.~Fabes, and Z.~Zhao.
\newblock Conditional gauge and potential theory for the {S}chr\"odinger
  operator.
\newblock {\em Trans. Amer. Math. Soc.}, 307(1):171--194, 1988.

\bibitem{MR939365}
C.~Dellacherie and P.-A. Meyer.
\newblock {\em Probabilities and potential. {C}}, volume 151 of {\em
  North-Holland Mathematics Studies}.
\newblock North-Holland Publishing Co., Amsterdam, 1988.
\newblock Potential theory for discrete and continuous semigroups, Translated
  from the French by J. Norris.

\bibitem{MR2782245}
M.~Hahn and S.~Umarov.
\newblock Fractional {F}okker-{P}lanck-{K}olmogorov type equations and their
  associated stochastic differential equations.
\newblock {\em Fract. Calc. Appl. Anal.}, 14(1):56--79, 2011.

\bibitem{MR2507445}
T.~Jakubowski.
\newblock On combinatorics of {S}chr\"odinger perturbations.
\newblock {\em Potential Anal.}, 31(1):45--55, 2009.

\bibitem{MR2643799}
T.~Jakubowski and K.~Szczypkowski.
\newblock Time-dependent gradient perturbations of fractional {L}aplacian.
\newblock {\em J. Evol. Equ.}, 10(2):319--339, 2010.

\bibitem{MR1642818}
V.~Liskevich and Y.~Semenov.
\newblock Two-sided estimates of the heat kernel of the {S}chr\"odinger
  operator.
\newblock {\em Bull. London Math. Soc.}, 30(6):596--602, 1998.

\bibitem{MR2314753}
T.~J. Lyons, M.~Caruana, and T.~L{\'e}vy.
\newblock {\em Differential equations driven by rough paths}, volume 1908 of
  {\em Lecture Notes in Mathematics}.
\newblock Springer, Berlin, 2007.
\newblock Lectures from the 34th Summer School on Probability Theory,
  Saint-Flour, July 6--24, 2004, with an introduction by J. Picard.

\bibitem{MR2568272}
M.~Magdziarz.
\newblock Stochastic representation of subdiffusion processes with
  time-dependent drift.
\newblock {\em Stochastic Process. Appl.}, 119(10):3238--3252, 2009.

\bibitem{Meerschaert2011216}
M.~M. Meerschaert, E.~Nane, and P.~Vellaisamy.
\newblock Distributed-order fractional diffusions on bounded domains.
\newblock {\em Journal of Mathematical Analysis and Applications}, 379(1):216
  -- 228, 2011.

\bibitem{1993-KM-BR}
K.~S. Miller and B.~Ross.
\newblock {\em An introduction to the fractional calculus and fractional
  differential equations}.
\newblock A Wiley-Interscience Publication. John Wiley \& Sons Inc., New York,
  1993.

\bibitem{MR1978999}
Q.~S. Zhang.
\newblock A sharp comparison result concerning {S}chr\"odinger heat kernels.
\newblock {\em Bull. London Math. Soc.}, 35(4):461--472, 2003.

\end{thebibliography}
\end{document}